\documentclass{article}
\usepackage{latexsym,amssymb,amsmath, amsthm,graphicx}

\theoremstyle{plain}
\newtheorem{theorem}{Theorem}[section]
\newtheorem{lemma}[theorem]{Lemma}
\newtheorem{corollary}[theorem]{Corollary}
\newtheorem{remark}[theorem]{Remark}
\newtheorem{example}[theorem]{Example}
\newtheorem{definition}[theorem]{Definition}

\usepackage{geometry}

\usepackage[mathscr]{euscript}
\title{A brief note on the Karhunen-Lo\`{e}ve expansion}

\author{Alen~Alexanderian\thanks{North Carolina State University, Raleigh, NC, USA.
         {email:~alexanderian@ncsu.edu}}}

\newcommand{\R}{\mathbb{R}}

\newcommand{\eps}{\varepsilon}

\newcommand{\F}{\mathcal{F}}

\newcommand{\B}{\mathcal{B}}

\newcommand{\var}[1]{{\mathrm{Var}}\left[ {#1} \right]}
\newcommand{\ave}[1]{{\mathrm{E}}\left[ {#1} \right]}

\newcommand{\ip}[2]{\left\langle {#1}, {#2} \right\rangle}
\newcommand{\norm}[1]{||{#1}||}
\newcommand{\Norm}[2]{||{#1}||_{\scriptscriptstyle{#2}}}

\begin{document}

\maketitle 

% abstract:
\begin{abstract}
We provide a detailed derivation of the  Karhunen--Lo\`{e}ve expansion of a
stochastic process. We also discuss briefly Gaussian processes, and provide a simple
numerical study for the purpose of illustration.
\end{abstract}
\tableofcontents

\section{Introduction}
The purpose of this brief note is to provide a self-contained coverage of the
idea of the Karhunen--Lo\`{e}ve (KL) expansion of a stochastic process.  The
writing of this note was motivated by being exposed to the many applications of
the KL expansion in uncertainty propagation through dynamical systems with
random parameter functions; see e.g.,~\cite{knio,Ghanem,Xiu10,Smith13}.  Since
a clear and at the same time rigorous coverage of the KL exapnsion is not so
simple to find in the literature, here we provide a simple account of the
theoretical basis for the KL expansion, including a detailed proof of
convergence.  We will see that the KL expansion is obtained through an
interesting application of the Spectral Theorem for compact normal operators,
in conjunction with Mercer's theorem which connects the spectral representation
of a Hilbert-Schmidt integral operator to the corresponding Hilbert-Schmidt
kernel.

We begin by recalling some functional analytic basics on compact operators in
Section~\ref{sec:compact}. The material in that section are classical and can
be found in many standard textbooks on the subject; see e.g.,~\cite{naylorsell}
for an accessible presentation.  Next, Mercer's Theorem is recalled in
Section~\ref{sec:mercer}.  Then, we recall some basics regarding stochastic
processes in Section~\ref{sec:stoc-proc}. In that section, a basic result
stating the equivalence of mean-square continuity of a stochastic process and
the continuity of the corresponding autocorrelation function is mentioned also.
In Section~\ref{sec:KL}, we discuss in detail KL expansions of centered
mean-square continuous stochastic processes including a proof of convergence.
Finally, in Section~\ref{sec:example}, we provide a numerical example where the
KL expansion of a Gaussian random field is studied. 

\section{Preliminaries on compact operators} \label{sec:compact}
Let us begin by recalling the notion of precompact and
relatively compact sets. 
\begin{definition}(Relatively Compact)\\
Let $X$ be a metric space; $A \subseteq X$ is relatively compact
in $X$, if $\bar{A}$ is compact in $X$.
\end{definition}
\begin{definition}(Precompact)\\
Let $X$ be a metric space; $A \subseteq X$ is precompact (also
called totally bounded) if
for every $\epsilon > 0$, there exist finitely many points
$x_1, \ldots , x_N$ in $A$ such that
$\displaystyle{\cup_1^N B(x_i, \epsilon)}$  covers $A$.
\end{definition}
The following Theorem shows that when we are working
in a complete metric space, precompactness and relative
compactness are equivalent.
\begin{theorem}
Let $X$ be a metric space. If $A \subseteq X$ is relatively compact
then it is precompact. Moreover, if $X$ is complete then the converse
holds also.
\end{theorem}

Then, we define a compact operator as below.
\begin{definition} \label{compact}
Let $X$ and $Y$ be two normed linear spaces and $T:X \to Y$ a linear
map between $X$ and $Y$. $T$ is called a compact operator
if for all bounded sets $E \subseteq X$, $T(E)$ is relatively
compact in $Y$.
\end{definition}

By the above definition~\ref{compact}, if $E \subset X$ is a
bounded set, then $\overline{T(E)}$ is compact in $Y$.
The following basic result shows a couple of different ways of 
looking at compact operators.
\begin{theorem} \label{thm:compact-basic}
Let $X$ and $Y$ be two normed linear spaces; suppose $T:X \to Y$,
is a linear operator. Then the following are equivalent.
\begin{enumerate}
\item $T$ is compact.
\item The image of the open unit ball under $T$ is relatively compact in $Y$.
\item For any bounded sequence $\{x_n\}$ in $X$, there exist
      a subsequence $\{Tx_{n_k}\}$ of $\{Tx_n\}$ that converges
      in $Y$.
\end{enumerate}
\end{theorem}

Let us denote by $B[X]$ the set of all bounded 
linear operators on a normed linear space space $X$:
\[
  B[X] = \{ T:X \to X | \text{ $T$ is a bounded linear transformation.} \}.
\]
Note that equipped by the operator norm $B[X]$ is a normed linear 
space. 
It is simple to show that compact operators
form a subspace of $B[X]$. The following  
result (cf.~\cite{naylorsell} for a proof) shows that the set of compact normal 
operators is in fact a closed subspace of $B[X]$.  
\begin{theorem} \label{thm:closed}
Let $\{ T_n \}$ be a sequence of compact  
operators on a normed linear space $X$. Suppose
$T_n \to T$ in $B[X]$. Then, $T$ is also a
compact operator. 
\end{theorem}
Another interesting fact regarding compact linear operators is that they form
an ideal of the ring of bounded linear mappings $B[X]$. This follows from the
following basic result whose simple proof is also included 
for reader's convenience.
\begin{lemma}
Let $X$ be a normed linear space, and let $T$ and $S$ be in 
$B[X]$. If $T$ is compact, then so are $ST$ and $TS$. 
\end{lemma}
\begin{proof}
Consider the mapping $ST$. Let $\{ x_n \}$ be a bounded sequence
in $X$. Then, by Theorem~\ref{thm:compact-basic}(3), there exists 
a subsequence $\{Tx_{n_k}\}$ of $\{Tx_n\}$ that converges
in $X$: $Tx_{n_k} \to y^* \in X$.
Now, since $S$ is continuous, it follows that 
$STx_{n_k} \to S(y^*)$; that is, 
$\{STx_{n_k}\}$ converges in $X$ also, and so $ST$ is compact.
To show $TS$ is compact, take a bounded sequence $\{ x_n \}$
in $X$ and note that $\{Sx_n\}$ is bounded also (since $S$ is
continuous). Thus, again by Theorem~\ref{thm:compact-basic}(3), 
there exists a subsequence $\{TSx_{n_k}\}$ which converges in $X$,
and thus, $TS$ is also compact.\qedhere 
\end{proof}

\begin{remark} \label{rmk:compact-inv}
A compact linear operator of an infinite
dimensional normed linear space is not invertible
in $B[X]$. To see this, suppose that $T$ has an inverse $S$ in 
$B[X]$. Now, applying the previous
Lemma, we get that $I = TS = ST$ is also compact. However, 
this implies that the closed unit ball in $X$ is compact, 
which is not possible since we assumed $X$ is infinite
dimensional.
(Recall that the closed unit ball in a normed linear space
$X$ is compact if and only if $X$ is finite dimensional.) 
\end{remark}

\subsection{Hilbert-Schmidt operators}
Let $D \subset \R^n$ be a bounded domain. 
We call a function  
$k:D \times D \to \R$ a Hilbert-Schmidt
kernel if 
\[
   \int_D \int_D |k(x,y)|^2 \, dx\,dy < \infty,
\]
that is, $k \in L^2(D \times D)$  
(note that one special case is when $k$ is a continuous
function on $D \times D$). 
Define the 
integral operator $K$ on $L^2(D)$, $K: u \to K u$ 
for $u \in L^2(D)$, by 
\begin{equation}\label{equ:hilb}
   [K u](x) = \int_D k(x,y)u(y) \, dy.  
\end{equation}
It is simple to show that $K$ is a bounded operator 
on $L^2(D)$. Linearity is clear. As for boundedness, we 
note that for every $u \in L^2(D)$, 
\begin{align*}
\|Ku\|_{L^2(D)}^2 = \int_D \Big| (K u)(x) \Big|^2 \, dx
&=
\int_D \Big| \int_D k(x,y) u(y) \, dy \Big|^2 \, dx
\\
&\leq \int_D \Big(\int_D |k(x,y)|^2\, dy\Big) 
      \Big(\int_D |u(y)|^2 \, dy\Big) \, dx
\qquad \text{(Cauchy-Schwarz)}
\\
&= \Norm{k}{L^2(D\times D)}\Norm{u}{L^2(D)} < \infty.
\end{align*} 

An integral operator $K$ as defined above is
called a \emph{Hilbert-Schmidt operator}. 
The following result which is usually proved using
Theorem~\ref{thm:closed} is very useful. 
\begin{lemma} \label{lem:compactop}
Let $D$ be a bounded domain in $R^n$ and let $k \in L^2(D \times D)$
be a Hilbert-Schmidt kernel. Then, the integral operator
$K : L^2(D) \to L^2(D)$ given by 
$[K u](x) = \int_D k(x, y) u(y) \, dy$ is a compact operator. 
\end{lemma}

\subsection{Spectral theorem for compact self-adjoint operators}
Let $H$ be a real Hilbert space with inner product 
$\ip{\cdot}{\cdot}:H \times H \to \R$. 
A linear operator $T:H \to H$ is called 
self adjoint if 
\[
  \ip{Tx}{y} = \ip{x}{Ty}, \quad \forall x, y \in H.
\]
\begin{example}
Let us consider a Hilbert-Schmidt
operator $K$ on $L^2([a,b])$ as in~\eqref{equ:hilb}
(where for simplicity we have taken $D = [a, b] \subset \R$).   
Then, 
it is simple to show that $K$ is self-adjoint
if and only if $k(x,y) = k(y,x)$ on $[a,b] \times [a,b]$.
\end{example}
A linear operator $T:H \to H$, is called positive 
if $\ip{Tx}{x} \ge 0$ for all $x$ in $H$. Recall that a 
scalar $\lambda \in \R$ is called an eigenvalue of $T$ if there exists
a non-zero $x \in H$ such that $Tx = \lambda x$. Note that 
the eigenvalues of a positive operator are necessarily 
non-negative.

Compact self-adjoint operators on infinite dimensioal 
Hilbert spaces resemble many properties of the symmetric matrices. 
Of particular interest is the spectral decomposition 
of a compact self-adjoint operator as given by the 
following:
\begin{theorem}
Let $H$ be a (real or complex) Hilbert space and
let $T:H \to H$ be a compact self-adjoint operator. Then, 
$H$ has an orthonormal basis $\{e_i\}$ of eigenvectors of $T$
corresponding to eigenvalues $\lambda_i$. In addition, the following
holds:
\begin{enumerate}
\item The eigenvalues $\lambda_i$ are real having zero as the 
only possible point of accumulation.  
\item The eigenspaces corresponding to distinct eigenvalues
are mutually orthogonal.
\item The eigenspaces corresponding to non-zero eigenvalues are finite-dimensional.
\end{enumerate}
\end{theorem} 
In the case of a positive compact self-adjoint operator, 
we know that the eigenvalues are non-negative. Hence, we may
order the eigenvalues as follows
\[
   \lambda_1 \ge \lambda_2 \ge ... \ge 0.
\]
\begin{remark}
Recall that for a linear operator $A$ on a finite dimensional linear 
space, we define 
its spectrum $\sigma(A)$ as the set of its eigenvalues. On the other hand, 
for a linear operator $T$ on an infinite dimensional (real) normed linear space
the spectrum $\sigma(T)$ of $T$ is defined by,
\[
   \sigma(T) = \{ \lambda \in \R : T - \lambda I \text{ is not invertible in } B[X]\},
\] 
and $\sigma(T)$ is the disjoint union of the point spectrum (set of eigenvalues), 
contiuous spectrum, and residual spectrum (see~\cite{naylorsell} for details).
As we saw in Remark~\ref{rmk:compact-inv}, a compact operator $T$ on an infinite dimensional
space $X$ cannot be invertible in $B[X]$; therefore, we always have $0 \in \sigma(T)$.
However, not much can be said on whether $\lambda = 0$ is in point spectrum 
(i.e. an eigenvalue) or the other parts of the spectrum. 
\end{remark}
\section{Mercer's Theorem} \label{sec:mercer}
Let $D = [a, b] \subset \R$. We have seen that given 
a continuous kernel $k:D\times D \to \R$, we can define 
a Hilbert-Schmidt operator through~\eqref{equ:hilb} which 
is compact and has a complete set of eigenvectors in $L^2(D)$. 
The following result by Mercer provides a series representation
for the kernel $k$ based on spectral representation of the 
corresponding Hilbert-Schmidt operator $K$. 
A proof of this result can be found for example in~\cite{Gohberg}.
\begin{theorem}[Mercer] \label{thm:mercer}
Let $k:D \times D \to \R$ be a continuous 
function, where $D = [a,b] \subset \R$.
Suppose further that the corresponding Hilbert-Schmidt operator
$K:L^2(D) \to L^2(D)$ given by~\eqref{equ:hilb} 
is postive. If $\{ \lambda_i \}$ and $\{ e_i \}$ 
are the eigenvalues and eigenvectors of $K$,
then for all $s, t \in D$,
\begin{equation} \label{equ:kernel-dec}
k(s, t) = \sum_i \lambda_i e_i(s) e_i(t),
\end{equation}
where convergence is absolute and uniform on $D\times D$. 
\end{theorem}
\section{Stochastic processes} \label{sec:stoc-proc}
In what follows we consider a probability space $(\Omega, \F, P)$,
where $\Omega$ is a sample space, $\F$ is an appropriate $\sigma$-algebra
on $\Omega$ and $P$ is a probability measure. A real valued
random
variable $X$ on $(\Omega, \F, P)$ is an $\F/\B(\R)$-measurable
mapping $X: (\Omega, \F, P) \to (\R, \B(\R))$.
The expectation and variance of a random variable $X$ is denoted by,
\[
   \ave{X} := \int_\Omega X(\omega) \, dP(\omega), 
   \quad \var{X} := \ave{ (X - \ave{X})^2}.
\]
$L^2(\Omega, \F, P)$
denotes the Hilbert space of (equivalence classes) of real valued
square integrable random variables on $\Omega$:
\[
   L^2(\Omega, \F, P) =
   \{ X : \Omega \to \R : \int_\Omega |X(\omega)|^2 \, dP(\omega) < \infty\}.
\]
with inner product, $\ip{X}{Y} = \ave{XY} = \int_\Omega X Y \, dP$
and norm $\norm{X} = \ip{X}{X}^{1/2}$.

Let $D \subseteq \R$,
a stochastic prcess is a mapping $X:D \times \Omega \to \R$, 
such that $X(t, \cdot)$ is measurable for every $t \in D$;
alternatively, we may define a stochastic process as 
a family of random variables, $X_t:\Omega \to \R$ with $t\in D$, 
and refer to $X$ as $\{X_t\}_{x \in D}$. 
Both of these points of view of a stochastic process are useful 
and hence we will be switching between them as appropriate.

A stochastic process is called \emph{centered} if $\ave{X_t} = 0$ for
all $t \in D$. 
Let $\{Y_t\}_{t \in D}$ be an arbitrary stochastic process. We note that
\[
   Y_t = \ave{Y_t} + X_t,
\]
where $X_t = Y_t - \ave{Y_t}$ and $\{X_t\}_{t \in D}$ is a centered stochastic 
process. Therefore, without loss of generality, we will focus our attention
to centered stochastic processes. 

We say a stochastic process is \emph{mean-square continuous}
if 
\[
   \displaystyle \lim_{\eps \to 0} \ave{(X_{t+\eps} - X_t)^2} = 0. 
\]
The following definition is also useful.
\begin{definition}[Realization of a stochastic process]
Let $X:D \times \Omega \to \R$ be a stochastic process. 
For a fixed $\omega \in \Omega$, 
we define $\hat{X} : D \to \R$ by $\hat{X}(t) = X_t(\omega)$. 
We call $\hat{X}$ a realization of the stochastic process. 
\end{definition}

For more details on theory of stochastic processes please
consult~\cite{Williams,RogersWilliamsV1,RogersWilliamsV2}. 

\subsection{Autocorrelation function of a stochastic process}
The autocorrelation function of a stochastic process $\{X_t\}_{t \in D}$ is given by 
$R_X:D \times D \to \R$ defined through
\[
R_X(s, t) = \ave{X_s X_t}, \quad s, t \in D.
\]
The following well-known result states that for a stochastic process  the
continuity of its autocorrelation function is a necessary and sufficient
condition for the mean-square continuity of the process.

\begin{lemma}\label{lem:cont}
A stochastic process $\{X_t\}_{t \in [a,b]}$ is mean-square continuous if and
only if its autocorrelation function $R_X$ is continuous
on $[a,b]\times[a,b]$. 
\end{lemma}
\begin{proof}
Suppose $R_X$ is continuous, and note that
\[
   \ave{(X_{t+\eps} - X_t)^2} = \ave{X_{t+\eps}^2} - 
                                2\ave{X_{t+\eps}X_t} + 
                                \ave{X_t^2} 
                              = R_X(t+\eps, t+\eps)
                                - 2 R_X(t+\eps,t) + R_X(t,t).
\]
Therefore, since $R_X$ is continuous,
\[
   \lim_{\eps \to 0} 
   \ave{(X_{t+\eps} - X_t)^2} = 
   \lim_{\eps \to 0} R_X(t+\eps, t+\eps) - 2 R_X(t+\eps,t) + R_X(t,t)
   = 0. 
\]
That is $X_t$ is mean-square continuous. 
Conversely, if $X_t$ is mean-square continous we proceed as follows: 
\begin{align*}
&|R_X(t+\eps, s+\nu) - R_X(t, s)| = 
               |\ave{X_{t+\eps} X_{s+\nu}} - \ave{X_t X_s}|\\
               &= 
               \Big|\ave{(X_{t+\eps}-X_t)(X_{s+\nu}-X_s)}  
               +\ave{(X_{t+\eps} - X_t)X_s} + \ave{(X_{s+\nu} - X_s)X_t}\Big|
               \\
               &\leq \Big|\ave{(X_{t+\eps}-X_t)(X_{s+\nu}-X_s)}\Big|  
               +\Big|\ave{(X_{t+\eps} - X_t)X_s}\Big| + 
                \Big|\ave{(X_{s+\nu} - X_s)X_t}\Big|
               \\
               &\leq 
               \ave{(X_{t+\eps}-X_t)^2}^{1/2}\ave{(X_{s+\nu}-X_s)^2}^{1/2}
               +\ave{(X_{t+\eps} - X_t)}^{1/2}\ave{X_s^2}^{1/2}\\ 
               &\hspace{6.75cm}+ \ave{(X_{s+\nu} - X_s)^2}^{1/2}\ave{X_t^2}^{1/2}, 
\end{align*}
where the last inequality follows from Cauchy-Schwarz inequality. 
Thus, we have, 

\begin{multline}
|R_X(t+\eps, s+\nu) - R_X(t, s)| 
          \leq 
               \ave{(X_{t+\eps}-X_t)^2}^{1/2}\ave{(X_{s+\nu}-X_s)^2}^{1/2}\\ 
               +\ave{(X_{t+\eps} - X_t)}^{1/2}\ave{X_s^2}^{1/2}
               + \ave{(X_{s+\nu} - X_s)^2}^{1/2}\ave{X_t^2}^{1/2},
\end{multline}
and therefore, by mean-square continuity of $X_t$ we have that 
\[
\lim_{(\eps, \nu) \to (0,0)} |R_X(t+\eps, s+\nu) - R_X(t, s)| = 0.
\qedhere\]
\end{proof}

\section{Karhunen--Lo\`{e}ve expansion} \label{sec:KL}
Let $D \subseteq \R$. In this section, we assume that 
$X : D \times \Omega \to \R$ is a  
\emph{centered} \emph{mean-square continuous} stochastic process
such that $X \in L^2(D \times \Omega)$. With the technical tools
from the previous sections, we are now ready to derive the
KL expansion of $X$. 
 
Define the integral operator $K:L^2(D) \to L^2(D)$ by 
\begin{equation} \label{equ:autocorr}
   [K u](s) = \int_D k(s,t) u(t) \, dt, \quad k(s,t) = R_X(s,t),
\end{equation} 
The following lemma summarizes the properties 
of the operator $K$. 
\begin{lemma}
Let $K:L^2(D) \to L^2(D)$ be as in~\eqref{equ:autocorr}. Then
the following hold:
\begin{enumerate}
\item $K$ is compact. 
\item $K$ is positive
\item $K$ is self-adjoint.
\end{enumerate}
\end{lemma}
\begin{proof}

(1) Since the process $X$ is mean-square continuous, 
Lemma~\ref{lem:cont} implies that $k(s,t) = R_X(s,t)$ is
continuous. Therefore, by Lemma~\ref{lem:compactop}, 
$K$ is compact.   

(2) We need to show $\ip{K u}{u} \ge 0$ for every $u \in L^2(D)$, where
$\ip{\cdot}{\cdot}$ denotes the $L^2(D)$ inner product. 
\begin{align*}
   \ip{K u}{u} = \int_D Ku(s) u(s) \, ds  
               &=\int_D \Big(\int_D k(s,t)u(t) \, dt\Big) u(s)\,ds
               \\
               &=\int_D \Big(\int_D \ave{X_s X_t} u(t) \, dt\Big) u(s)\,ds
               \\
               &=\ave{\int_D \int_D X_s X_t u(t) u(s) \, dt \,ds}
               \\
               &=\ave{ \Big(\int_D X_s u(s) \,ds\Big) 
                       \Big(\int_D X_t u(t) \, dt\Big)}
               \\
               &=\ave{ \Big(\int_D X_t u(t) \,dt\Big)^2} \geq 0,
\end{align*}
where we used Fubini's Theorem to interchange integrals.  

(3) This follows trivially from $R_X(s,t) = R_X(t,s)$
and Fubini's theorem:
\[
   \ip{K u}{v} = \int_D Ku(s) v(s) \, ds  
               =\int_D \Big(\int_D k(t,s)v(s) \, ds\Big) u(t)\,dt
               =\ip{u}{Kv}.
\qedhere\]
\end{proof}

Now, let $K$ be defined as in~\eqref{equ:autocorr} the previous lemma
allows us to invoke the spectral theorem for compact self-adjoint operators
to conclude that $K$ has a complete set of eigenvectors $\{e_i\}$ in 
$L^2(D)$ and real eigenvalues $\{\lambda_i\}$:
\begin{equation}\label{equ:eigen}
 K e_i = \lambda_i e_i. 
\end{equation}
Moreover, since $K$ is positive, the
eigenvalues $\lambda_i$ are non-negative (and have zero as 
the only possible accumulation point). Now, the stochastic 
process $X$ which we fixed in the beginning of this section is assumed
to be square integrable on $D \times \Omega$ and thus, we may use
the basis $\{e_i\}$ of $L^2(D)$ to expand $X_t$ as follows,
\begin{equation} \label{equ:basic-conv}
   X_t = \sum_i x_i e_i(t),
\quad x_i = \int_D X_t e_i(t) \, dt
\end{equation}
The above
equality is to be understood in mean square sense. To be most specific, 
at this point we have that the realizations $\hat{X}$ of the stochastic
process $X$ admit the expansion  
\[
   \hat{X} = \sum_i x_i e_i
\]
where the convergence is in $L^2(D \times \Omega)$. We will see shortly that 
the result is in fact stronger, and we have 
\[
   \lim_{N \to \infty} \ave{ \Big(X_t - \sum_{i=1}^N x_i e_i(t)\Big)^2 } = 0, 
\]
uniformly in $D$, and thus, as a consequence, we have that~\eqref{equ:basic-conv} 
holds for all $t \in D$. 
Before proving this, we examine 
the coefficients $x_i$ in~\eqref{equ:basic-conv}. Note that $x_i$ 
are random variables on $\Omega$. The following 
lemma summarizes the properties of the coefficients $x_i$.
\begin{lemma} \label{lem:coeffs}
The coefficients $x_i$ in~\eqref{equ:basic-conv} satisfy the following: 
\begin{enumerate}
\item $\displaystyle \ave{x_i} = 0$
\item $\displaystyle \ave{x_i x_j} = \delta_{ij} \lambda_j$.
\item $\displaystyle \var{x_i} = \lambda_i$.
\end{enumerate}
\end{lemma}
\begin{proof}
To see the first assertion note that 
\begin{align*}
\ave{x_i} &= \ave{ \int_D X_t e_i(t) \, dt } \\
          &= \int_\Omega \int_D X_t(\omega) e_i(t) \, dt \, dP(\omega) \\
          &= \int_D \int_\Omega X_t(\omega) e_i(t) \, dP(\omega) \, dt 
             \quad \text{(Fubini)}\\
          &= \int_D \ave{X_t} e_i(t) \, dt = 0, 
\end{align*}
where the last conclusion follows from $\ave{X_t} = 0$ ($X$ is a centered
process).
To see the second assertion, we proceed as follows
\begin{align*}
\ave{x_ix_j} &= \ave{ \Big(\int_D X_s e_i(s) \, ds\Big)
                      \Big(\int_D X_t e_j(t) \, dt\Big)} \\
             &= \ave{ \int_D \int_D X_s e_i(s) 
                      X_t e_j(t) \, ds \, dt} \\
             &= \int_D \int_D \ave{X_s X_t} e_i(s) e_j(t) \, ds \, dt \\
             &= \int_D \Big(\int_D k(s,t) e_j(t) \, dt\Big)e_i(s) \, ds \\
             &= \int_D [K e_j](s) e_i(s) \, ds \qquad \text{(from~\eqref{equ:autocorr})}\\
             &= \ip{Ke_j}{e_i}\\ 
             &= \ip{\lambda_j e_j}{e_i} 
             = \lambda_j \delta_{ij},
\end{align*}
where again we have used Fubini's Theorem to interchange integrals and
the last conclusion follows from orthonormality of eigenvectors of $K$.
The assertion (3) of the lemma follows easily from (1) and (2):
\[
   \var{x_i} = \ave{ (x_i - \ave{x_i})^2} = \ave{x_i^2} = \lambda_i.\qedhere  
\] 
\end{proof}

Now, we have the technical tools to prove the following:
\begin{theorem}[Karhunen--Lo\`{e}ve] \label{thm:KL}
Let $X:D \times \Omega \to \R$ be a centered mean-square
continuous stochastic process with $X \in L^2(\Omega \times D)$. 
There exist a basis $\{ e_i \}$ of $L^2(D)$ such that for 
all $t \in D$,
\[
\displaystyle X_t = \sum_{i=1}^\infty x_i e_i(t), \quad \text{ in } L^2(\Omega),
\]
where coefficients $x_i$ are 
given by 
$x_i(\omega) = \int_D X_t(\omega) e_i(t) \, dt$ and satisfy the following.
\begin{enumerate}
\item $\displaystyle \ave{x_i} = 0$
\item $\displaystyle \ave{x_i x_j} = \delta_{ij} \lambda_j$.
\item $\displaystyle \var{x_i} = \lambda_i$.
\end{enumerate}
\end{theorem}
\begin{proof}
Let $K$ be the Hilbert-Schmidt operator defined as in~\eqref{equ:autocorr}.
We know that $K$ has a complete set of eigenvectors $\{e_i\}$ in $L^2(D)$ and 
non-negative eigenvalues $\{ \lambda_i \}$. 
Note that $x_i(\omega) = \int_D X_t(\omega) e_i(t) \, dt$
satisfy the 
the properties (1)-(3) by Lemma~\ref{lem:coeffs}. 
Next, consider
\[
   \eps_n(t) := \ave{ \Big(X_t - \sum_{i=1}^n x_i e_i(t)\Big)^2 }.
\]
The rest of the proof amounts to showing 
$\displaystyle \lim_{n \to \infty}\eps_n(t) = 0$ uniformly (and
hence pointwise) in $D$. 
\begin{equation} \label{equ:exp-I}
\begin{aligned}
\eps_n(t) &= \ave{ \Big(X_t - \sum_{i=1}^n x_i e_i(t)\Big)^2 } \\ 
          &= \ave{ X_t^2} - 2 \ave{X_t \sum_{i=1}^n x_i e_i(t)} + 
                              \ave{\sum_{i,j=1}^n x_i x_j e_i(t)e_j(t)}
\end{aligned}
\end{equation}
Now, $\ave{X_t^2} = k(t,t)$ with $k$ as in~\eqref{equ:autocorr}, 
\begin{align}
\ave{X_t \sum_{i=1}^n x_i e_i(t)} &= 
   \ave{X_t \sum_{i=1}^n \big(\int_D X_s e_i(s) \, ds\big) e_i(t)} 
   \notag \\ 
   &= \sum_{i=1}^n \Big(\int_D \ave{X_t X_s} e_i(s) \, ds\Big) e_i(t) 
   \notag \\
   &= {\sum_{i=1}^n \Big(\int_D k(t,s) e_i(s) \, ds\Big) e_i(t)} 
   = {\sum_{i=1}^n [K e_i](t) e_i(t)} 
   = \sum_{i=1}^n \lambda_i e_i(t)^2. \label{equ:exp-II} 
\end{align}
Through a similar argument, we can show that
\begin{equation} \label{equ:exp-III}
\ave{\sum_{i,j=1}^n x_i x_j e_i(t)e_j(t)} = \sum_{i=1}^n \lambda_i e_i(t)^2
\end{equation}
Therefore, by~\eqref{equ:exp-I},~\eqref{equ:exp-II}, and~\eqref{equ:exp-III}
we have
\[
\eps_n(t) = k(t,t) - \sum_{i=1}^n \lambda_i e_i(t)e_i(t),
\]
invoking Theorem~\ref{thm:mercer} (Mercer's Theorem) we have 
\[
   \lim_{n\to\infty} \eps_n(t) = 0,
\]
uniformly; this completes the proof.  
\qedhere\end{proof}

\begin{remark}
Suppose $\lambda_k = 0$ for some $k$, and consider 
the coefficient $x_k$ in the expansion~\eqref{equ:basic-conv}. 
Then, we have by the above Theorem $\ave{x_k} = 0$ and
$\var{x_k} = \lambda_k = 0$, and therefore, $x_k = 0$.
That is, the coefficient $x_k$ corresponding to a zero
eigenvalue is zero. 
Therefore, only $x_i$ corresponding to postive eigenvalues $\lambda_i$
appear in KL expansion of a square integrable, centered, and 
mean-square continous stochastic process.
\end{remark}
In the view of the above remark, we can normalize the coefficients $x_i$ in 
a KL expansion and define $\xi_i = \frac{1}{\sqrt{\lambda_i}} x_i$. This leads 
to the following, more familiar, version of Theorem~\ref{thm:KL}.
\begin{corollary}
Let $X:D \times \Omega \to \R$ be a centered mean-square
continuous stochastic process with $X \in L^2(\Omega \times D)$.
There exist a basis $\{ e_i \}$ of $L^2(D)$ such that for
all $t \in D$,
\begin{equation}\label{equ:convL2}
X(t, \omega) = \sum_{i=1}^\infty \sqrt{\lambda_i} \xi_i(\omega) e_i(t)
\quad \text{ in } L^2(\Omega).
\end{equation}
where $\xi_i$ are centered mutually uncorrelated random variables with 
unit variance and are given by,
\[
\xi_i(\omega) = \frac{1}{\sqrt{\lambda_i}}\int_D X_t(\omega) e_i(t) \, dt.
\]
\end{corollary}

The KL expansion of a Gaussian process has the further property that $\xi_i$
are independent standard normal random variables (see e.g.~\cite{knio,Ghanem}).
The latter is a useful property in practical applications; for instance, this
is used extensively in the method of stochastic finite element~\cite{Ghanem}.
Moreover, in the case of a Gaussian process, the series representation
in~\eqref{equ:convL2} converges almost surely~\cite{Loeve77}. 

\section{A classical example}\label{sec:example}
Here we consider the KL decomposition of a Gaussian random field $X$, 
which is characterized by its variance $\sigma^2$ and an 
autocorrelation function $R_X(s, t)$ given by,
\begin{equation}\label{equ:kernel}
   R_X(s, t) = \sigma^2 \exp\Big(-\frac{|s-t|}{L_c}\Big).
\end{equation}
We show in Figure~\ref{fig:RX}(a) a plot of $R_X(s, t)$ 
over $[0, 1] \times [0, 1]$.
\begin{figure}[ht]\centering
\begin{tabular}{ccc}
\includegraphics[width=0.33\textwidth]{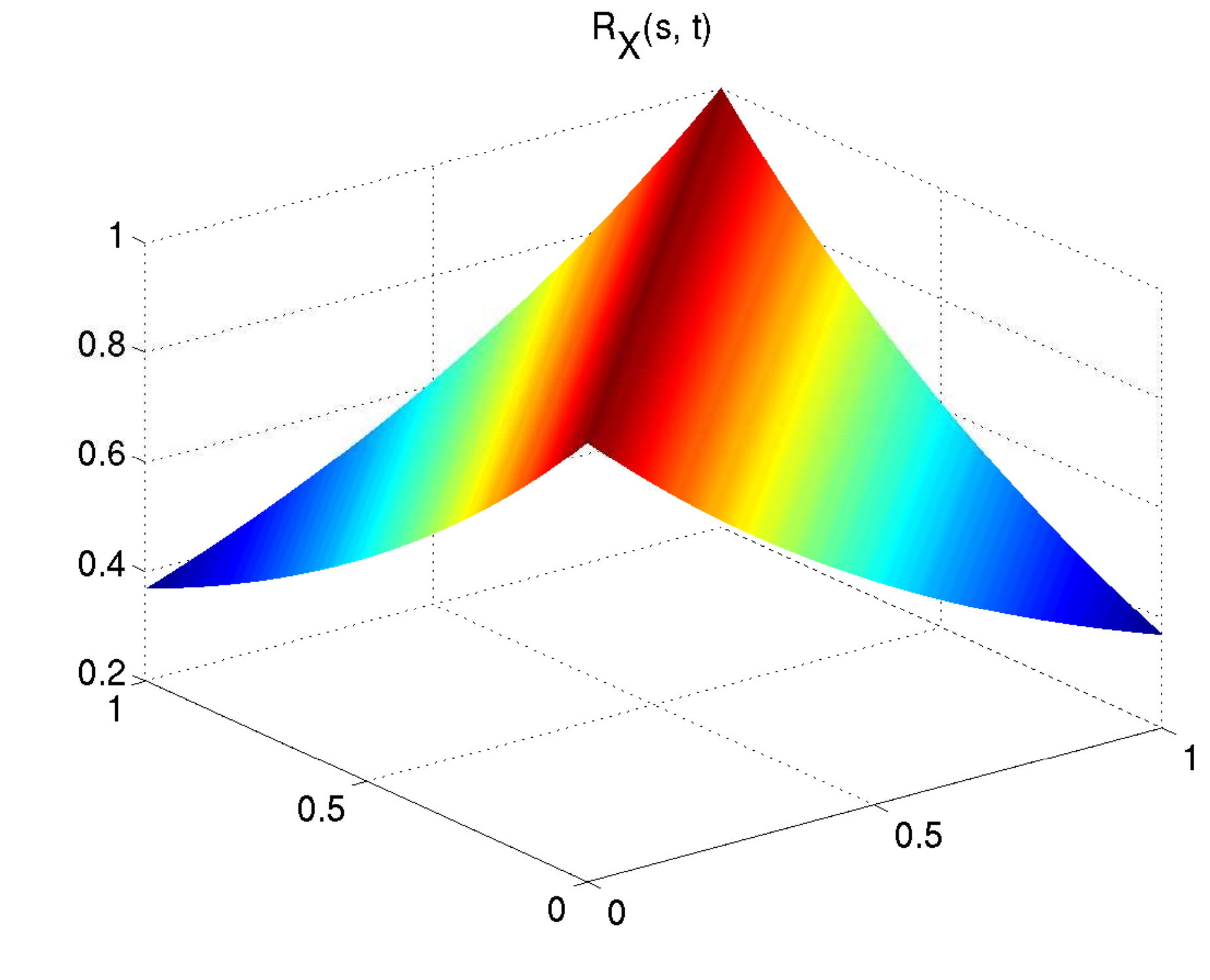}& 
\includegraphics[width=0.3\textwidth]{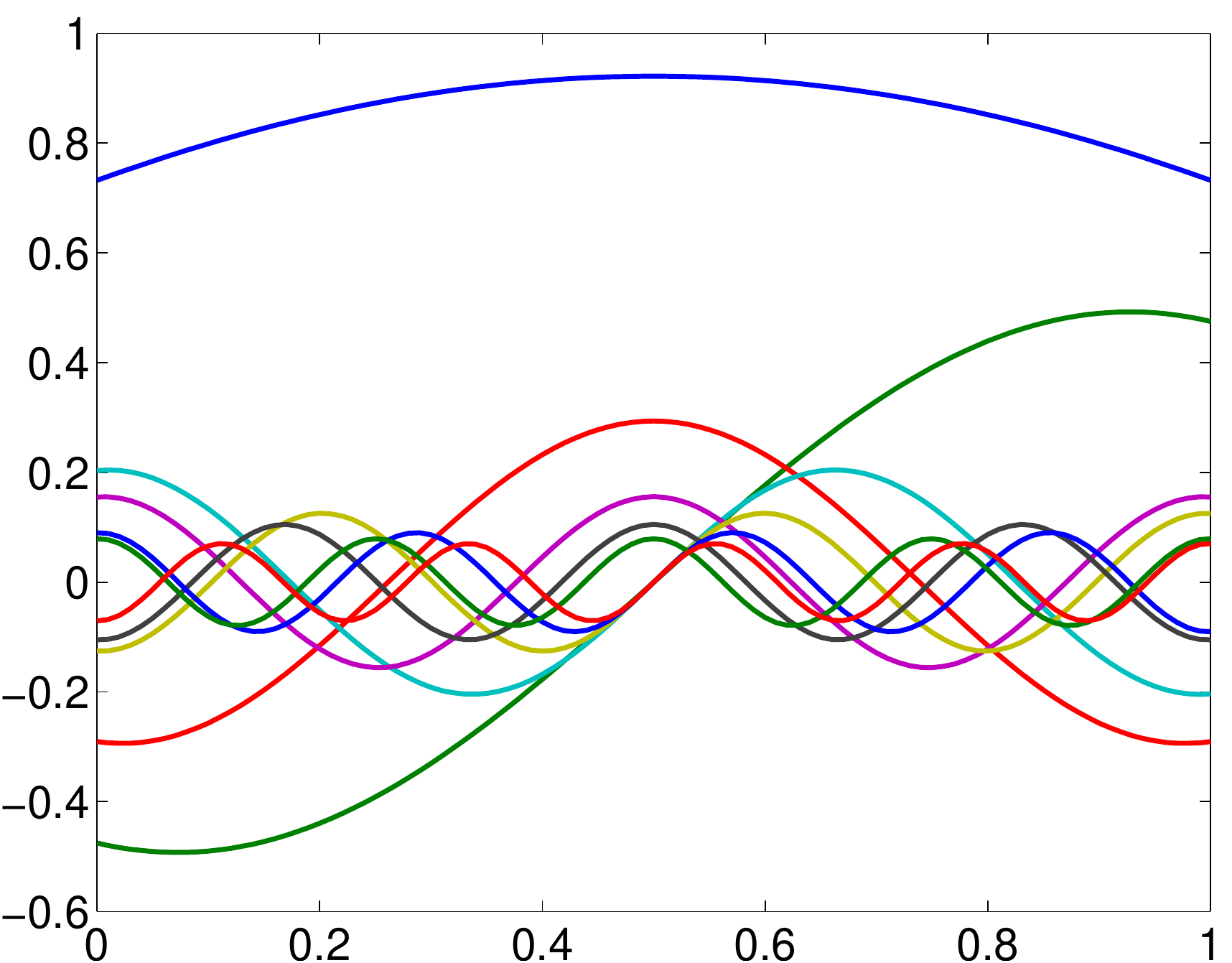} &
\includegraphics[width=0.3\textwidth]{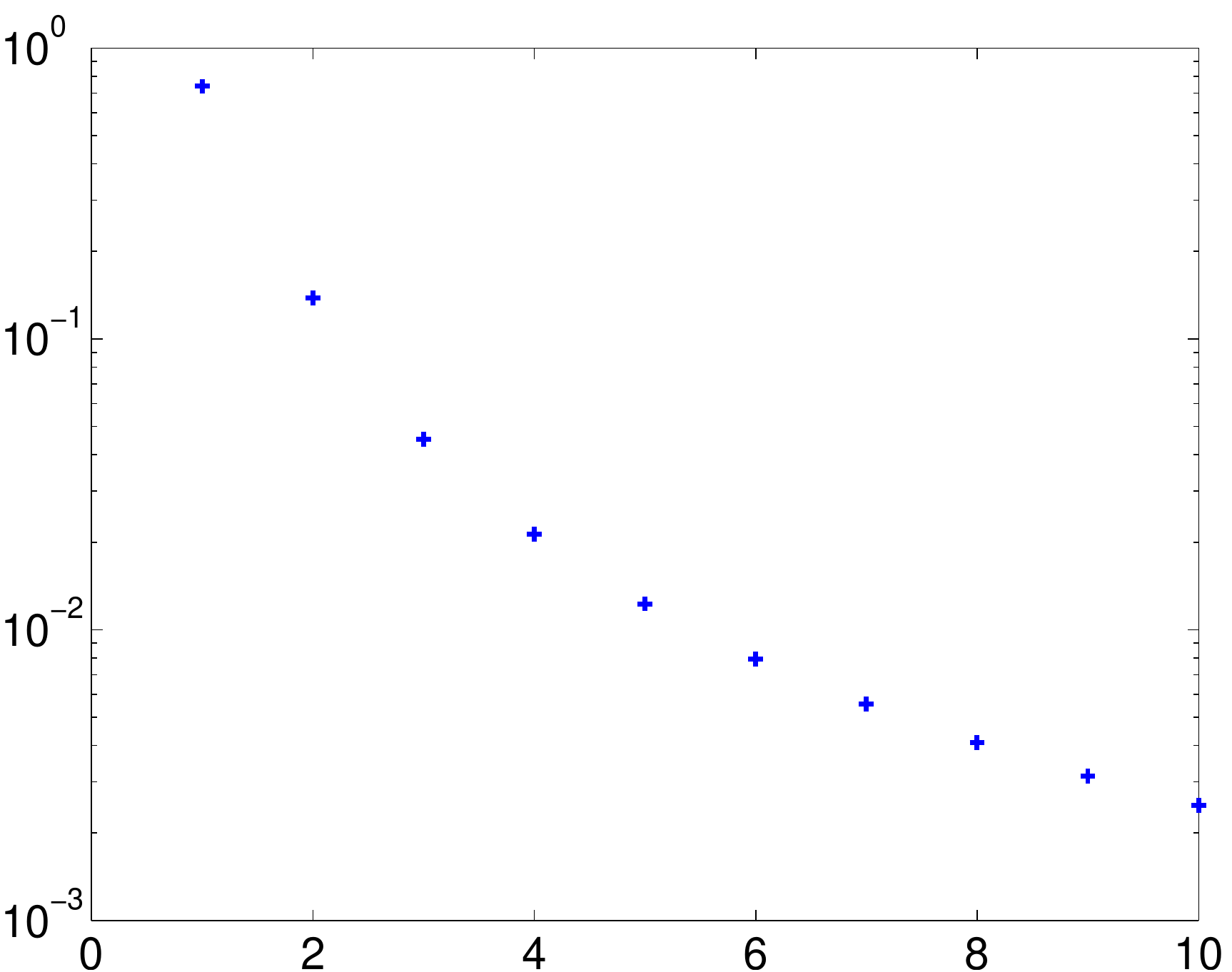}\\
(a) & (b) & (c)
\end{tabular}
\caption{
The autocorrelation function (a); 
the first few eigenfunctions of the integral operator whose kernel is given by the autocorrelation function (b), and 
the corresponding eigenvalues (c).}
\label{fig:RX}
\end{figure}
\paragraph{Spectral decomposition of the autocorrelation function}
For this particular example, the eigenfunctions $e_i(t)$ and
eigenvalues $\lambda_i$ can be computed analytically.
The analytic expression for eigenvalues and eigenvectors
can be found for example in~\cite{Ghanem,knio}.
We consider the case of $\sigma^2 = 1$ and $L_c = 1$ in~\eqref{equ:kernel}.
In Figure~\ref{fig:RX}(b)--(c), we show the first few eigenfunctions
and eigenvalues of the autocorrelation function defined in~\eqref{equ:kernel}.
To get an idea of how fast the approximation, 
\[
   R_{X}^N(s, t) = \sum_{i=1}^N \lambda_i e_i(s)e_i(t)
\]
converges to $R_X(s, t)$ we show in Figure~\ref{fig:conv}
the plots of $R^N_X(s, t)$ for $N = 2, 4, 6, 8$. 
In Figure~\ref{fig:compare}, we see that with $N = 6$, 
absolute error is bounded by $8 \times 10^{-2}$. 
\begin{figure}[ht]\centering
\begin{tabular}{cc}
\includegraphics[width=0.4\textwidth]{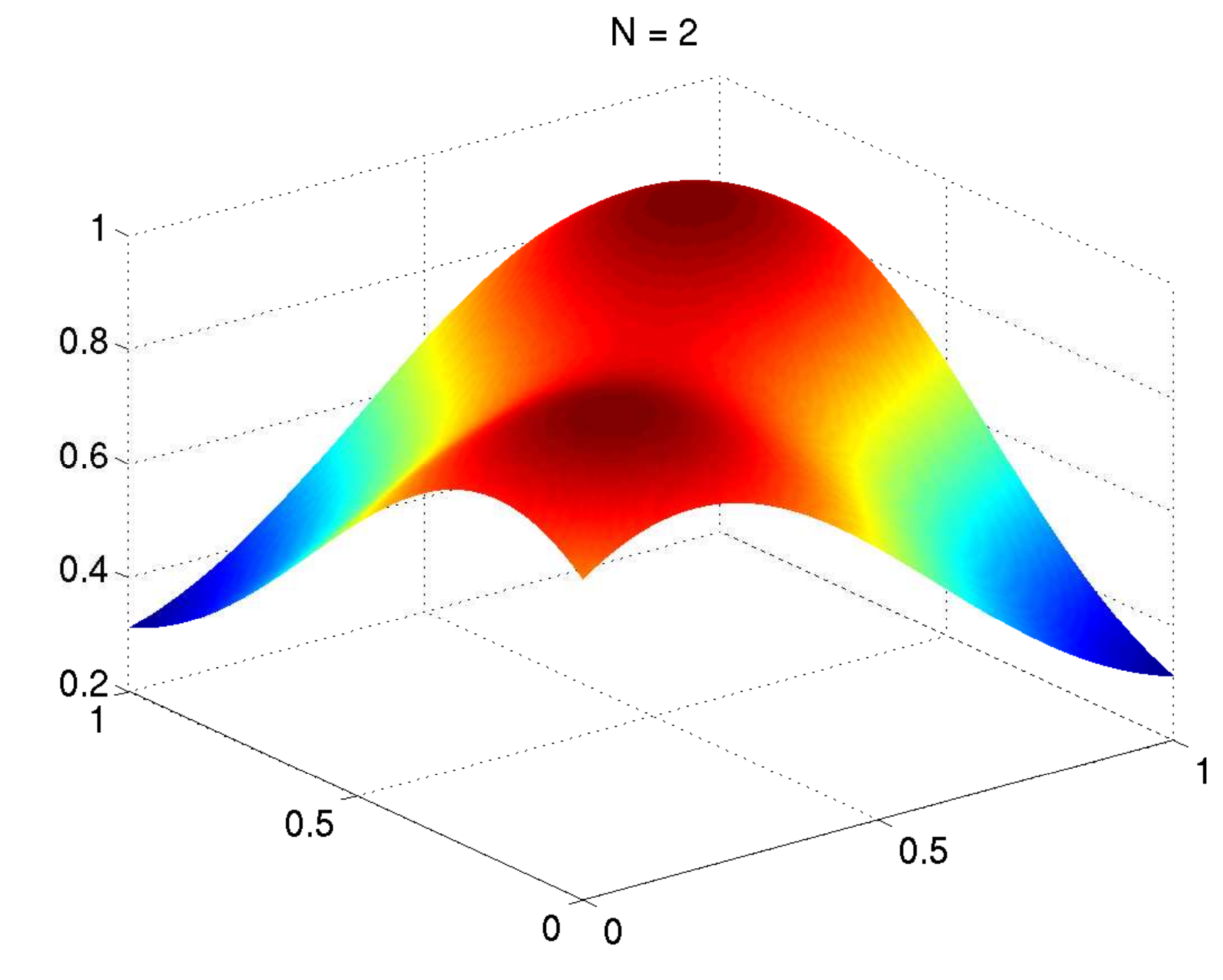} & 
\includegraphics[width=0.4\textwidth]{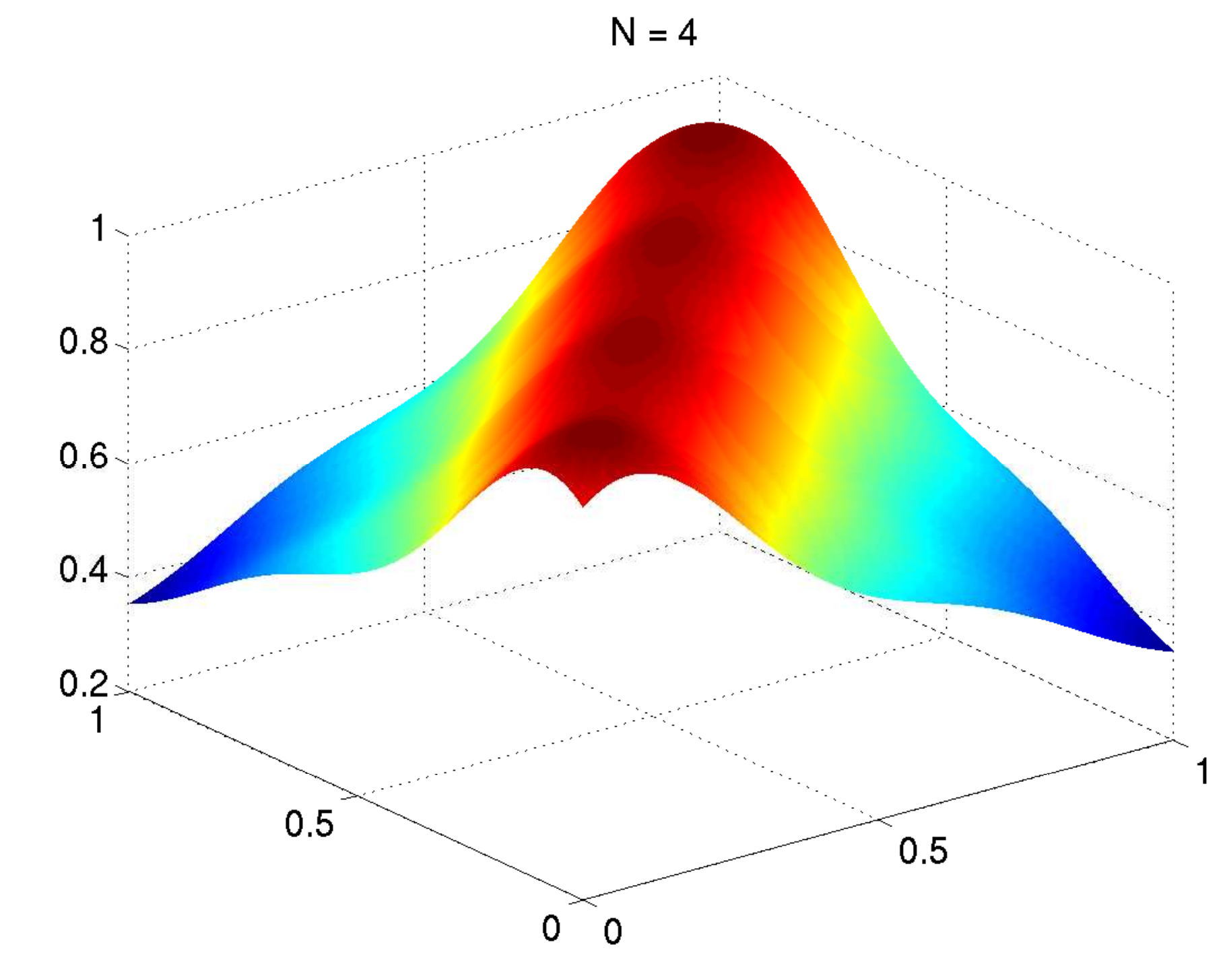} \\ 
\includegraphics[width=0.4\textwidth]{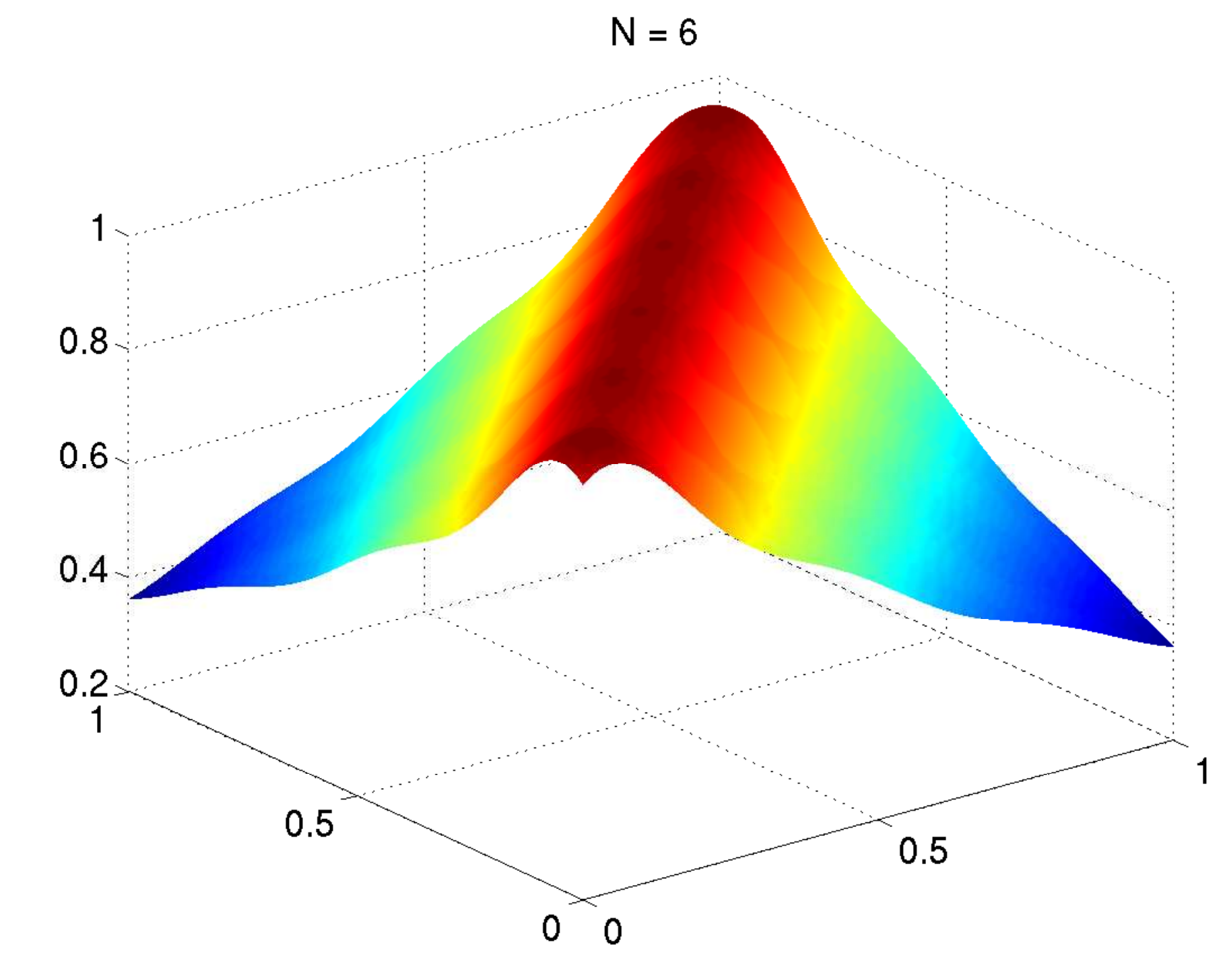} & 
\includegraphics[width=0.4\textwidth]{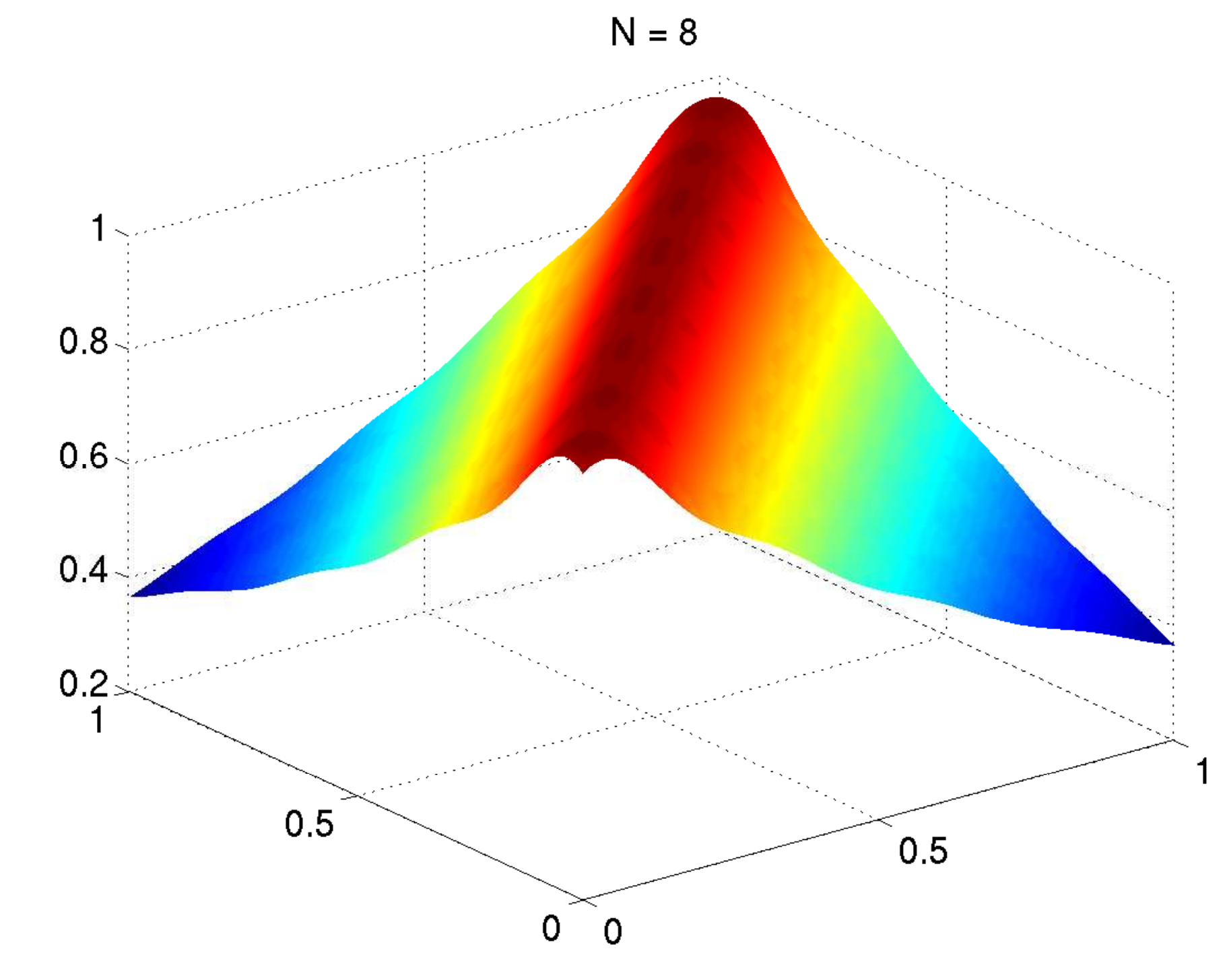}  
\end{tabular}
\caption{Improvements of the approximations to $R_X(s, t)$ as the
expansion order is increased.}
\label{fig:conv}
\end{figure}

\begin{figure}[ht]\centering
\begin{tabular}{ccc}
\includegraphics[width=0.33\textwidth]{kernel.pdf} &
\includegraphics[width=0.33\textwidth]{KernelN=6.pdf} &
\includegraphics[width=0.33\textwidth]{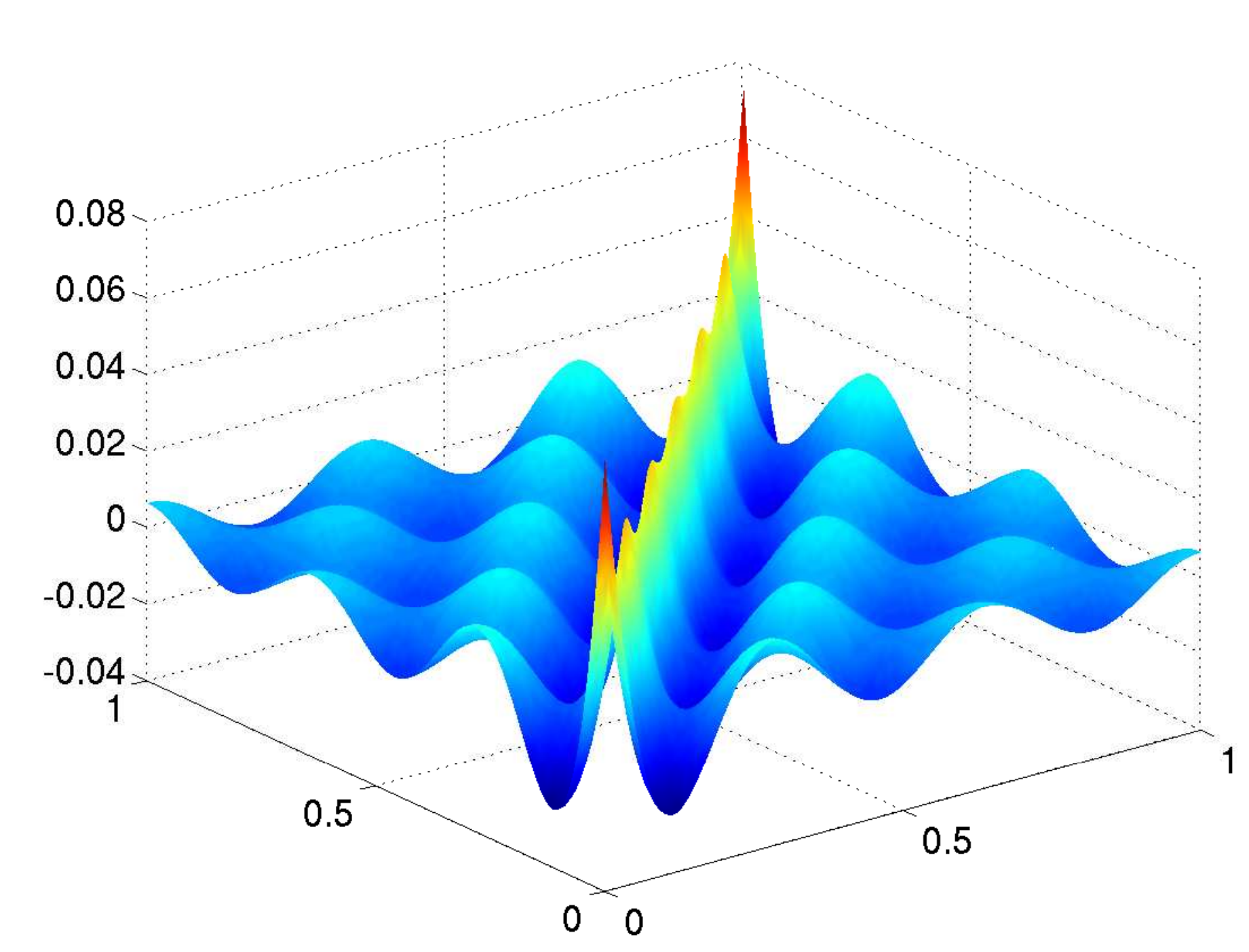} \\
(a) & (b) & (c)
\end{tabular}
\caption{(a) The autocorrelation function $R_X(s, t)$, (b) the approximation $R^N_X(s, t)$ with
$N = 6$, and (c) pointwise difference between $R_X(s, t)$ and $R_X^N(s, t)$ with $N = 6$.}
\label{fig:compare}
\end{figure}

\paragraph{Simulating the random field}
Having the eigenvalues and eigenfunctions of $R_X(t, \omega)$ at hand, we can simulate
the random field $X(t, \omega)$ with a truncated KL expansion, 
\[
X_\text{trunc}^{N}(t, \omega) := \sum_{i=1}^N \sqrt{\lambda_i} \xi_i(\omega) e_i(t).
\]
As discussed before, in this case, $\xi_i$ are independent standard normal
variables. In Figure~\ref{fig:simulate}(a), we plot a few realizations of the
truncated KL expansion of $X(t, \cdot)$, $t \in [0, 1]$ and in
Figure~\ref{fig:simulate}(b), we show the distribution of $X(t, \cdot)$ at $t =
1/2$ versus standard normal distribution. For this experiment we used a low oreder 
KL expansion with $N = 6$
terms.  
\vspace{-3mm}
\begin{figure}[ht]\centering
\begin{tabular}{cc}
\includegraphics[width=0.33\textwidth]{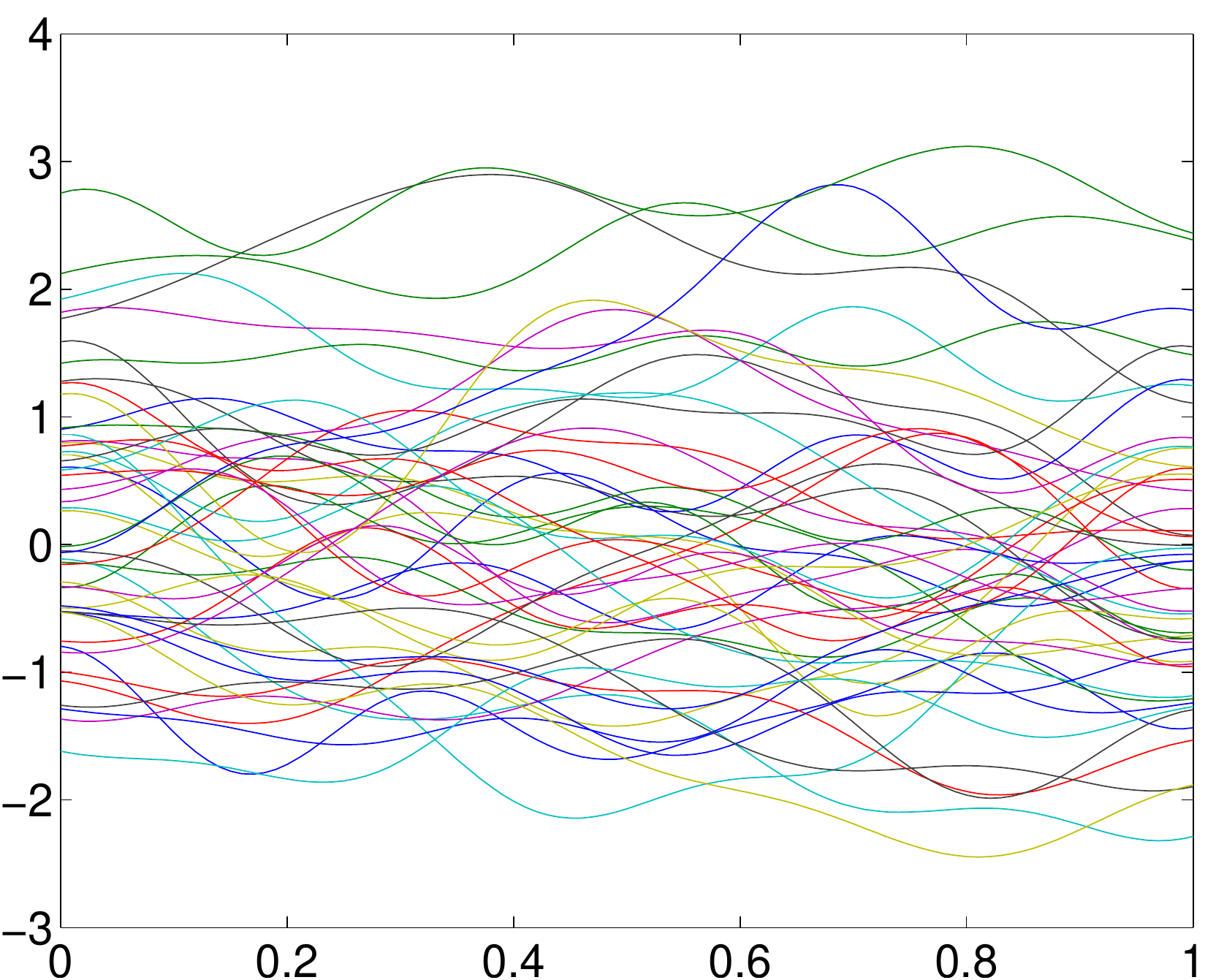} &
\includegraphics[width=0.33\textwidth]{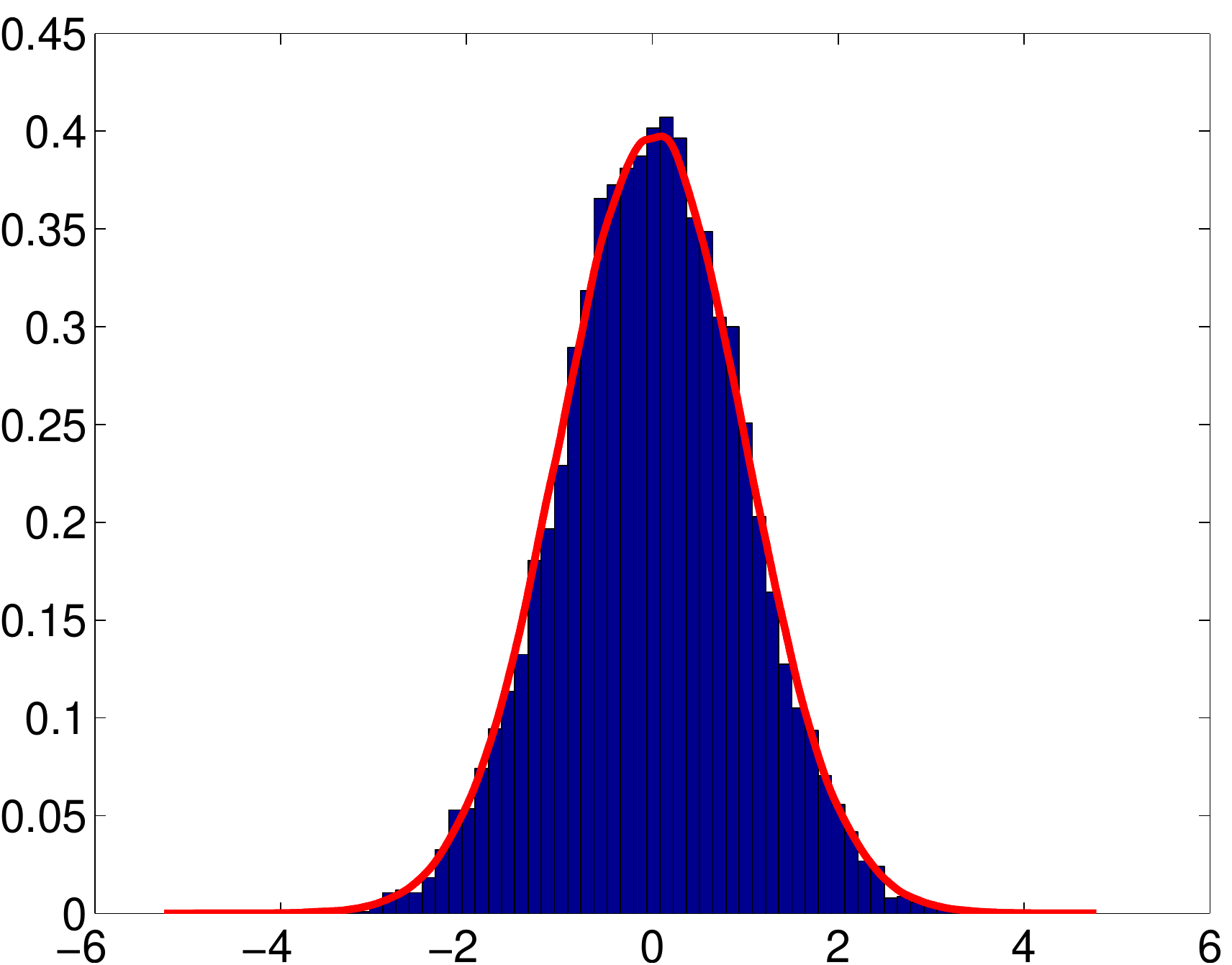}\\
(a) & (b)
\end{tabular}
\caption{(a) A few realizations of the random field $X(t, \cdot)$
approximated by  
a truncated KL expansion with $N = 6$ terms. 
(b) distribution of $X(t, \omega)$ at $t = 1/2$ (blue) versus a standard normal distribution 
(red).}
\label{fig:simulate}
\end{figure}

\paragraph{A final note regarding practical applications of KL expansions}
In practice, when using KL expansions to model uncertainties in mathematical
models, a premature \emph{a priori} truncation of the KL expansion 
could potentially lead to
misleading results, because the effect of the higher order oscillatory modes on the
output of a physical system could be significant.  Also, sampling such a
low-order KL expansion results in realizations of the random field that might
look artificially smooth; see for example the realizations of a low-order KL
expansion reported in Figure~\ref{fig:simulate}. 
In Figure~\ref{fig:simulate_ref} we illustrate the influence of the higher order modes on the
realizations of the truncated KL expansion, in the context of the same example;
in the figure, we consider two fixed realizations of the process, and for each realization
we plot $X_\text{trunc}^N(t, \omega)$ with successively larger values of $N$. 
\begin{figure}[ht]\centering
\includegraphics[width=0.4\textwidth]{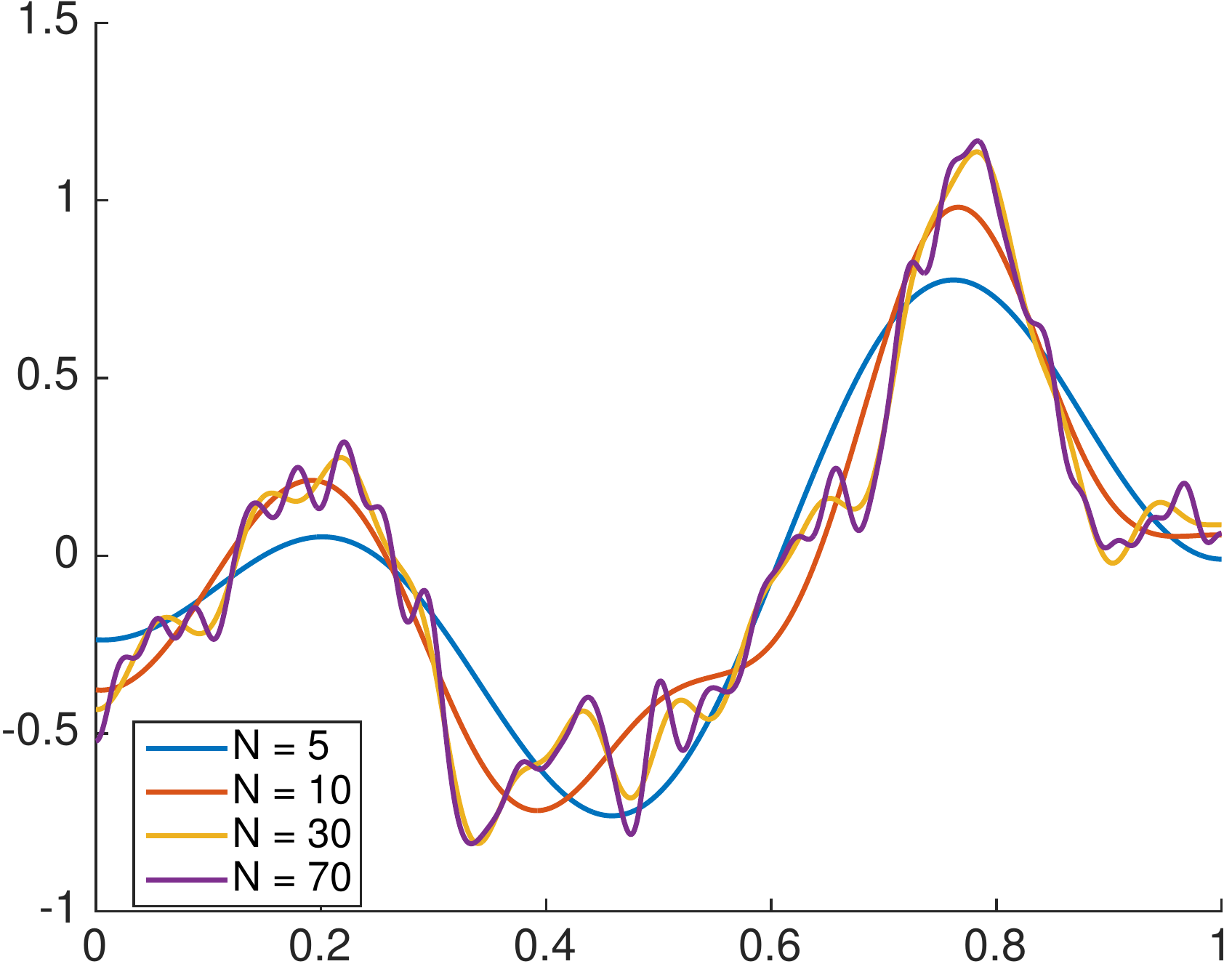} 
\includegraphics[width=0.4\textwidth]{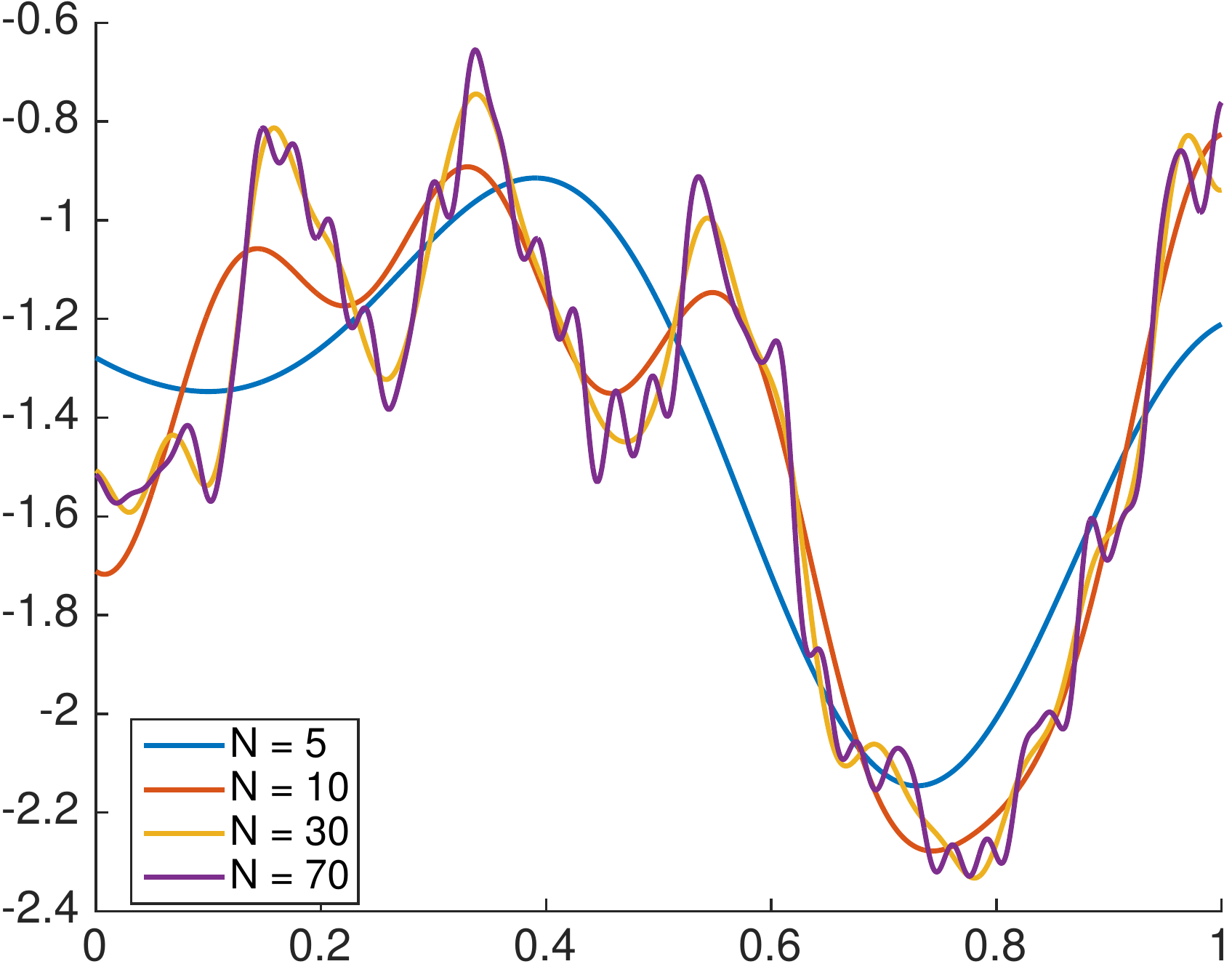} 
\caption{
Two realizations of the random field $X(t, \cdot)$ simulated via a truncated KL expansion. To see  
the influence of the higher order oscillations captured by higher order KL modes, we successively increase the truncation order $N$.}
\label{fig:simulate_ref}
\end{figure}

%-------------------------
\clearpage
\bibliographystyle{plain}
\bibliography{KL}

\def\cprime{$'$}
\begin{thebibliography}{10}

\bibitem{Ghanem}
Roger~G. Ghanem and Pol~D. Spanos.
\newblock {\em Stochastic finite elements: a spectral approach}.
\newblock Springer-Verlag New York, Inc., New York, NY, USA, 1991.

\bibitem{Gohberg}
Israel Gohberg, Seymour Goldberg, and M.~A. Kaashoek.
\newblock {\em Basic classes of linear operators}.
\newblock 2004.

\bibitem{knio}
Olivier~P. Le~Maitre and Omar~M. Knio.
\newblock {\em Spectral Methods for Uncertainty Quantification With
  Applications to Computational Fluid Dynamics}.
\newblock Scientific Computation. Springer, 2010.

\bibitem{Loeve77}
Michel Loeve.
\newblock {\em Probability theory {I}}, volume~45 of {\em Graduate Texts in
  Mathematics}.
\newblock New York, Heidelberg, Berlin: Springer-Verlag, 1977.

\bibitem{naylorsell}
Arch~W. Naylor and George Sell.
\newblock {\em Linear operator theory in engineering and science}.
\newblock Springer-Verlag, New York, 1982.

\bibitem{RogersWilliamsV1}
L.C.G. Rogers and David Williams.
\newblock {\em Diffusions, Markov Processes, and Martingales: Volume 1,
  Foundations}.
\newblock Cambridge University Press, 2000.

\bibitem{RogersWilliamsV2}
L.C.G. Rogers and David Williams.
\newblock {\em Diffusions, Markov processes and martingales: Volume 2, It{\^o}
  calculus}.
\newblock Cambridge university press, 2000.

\bibitem{Smith13}
Ralph~C. Smith.
\newblock {\em Uncertainty Quantification: Theory, Implementation, and
  Applications}, volume~12.
\newblock SIAM, 2013.

\bibitem{Williams}
David Williams.
\newblock {\em Probability with martingales}.
\newblock Cambridge Mathematical Textbooks. Cambridge University Press,
  Cambridge, 1991.

\bibitem{Xiu10}
Dongbin Xiu.
\newblock {\em Numerical methods for stochastic computations: a spectral method
  approach}.
\newblock Princeton University Press, 2010.

\end{thebibliography}

\end{document}